\documentclass[11pt]{amsart}
\usepackage{geometry}                
\usepackage{amssymb}
\usepackage{url}
\usepackage{lipsum}

\newtheorem{theorem}{Theorem}
\newtheorem{prop}[theorem]{Proposition}

\newcommand{\Rbb}{\mathbb{R}}
\newcommand{\Cbb}{\mathbb{C}}

\newcommand{\Real}{\text{Re}\,}
\newcommand{\Imag}{\text{Im}\,}
\newcommand{\calh}{\mathcal{H}}
\newcommand{\graph}{\text{graph}}
\newcommand{\vp}{\varphi}
\newcommand{\hdash}{\mbox{---------}}

\newcommand{\mmap}{\mathfrak{X}}
\newcommand{\Js}{\mathcal{J}}
\newcommand{\Sp}{\text{Sp}(2n,\mathbb{R})}

\newcommand{\graphch}{\graph_{\Cbb}\calh}

\newcommand{\sgap}{\vspace{12pt}}
\newcommand{\mgap}{\vspace{24pt}}
\newcommand{\important}[1]{\noindent\textbf{#1}}

\title[Linear Symplectomorphisms as $R$-Lagrangian Subspaces]{Linear Symplectomorphisms\\ as $R$-Lagrangian Subspaces}
\author{Chris Hellmann}
\author{Brennan Langenbach}
\author{Michael VanValkenburgh}

\begin{document}
\maketitle

\begin{abstract}
	The graph of a real linear symplectomorphism is an $R$-Lagrangian subspace of a complex symplectic vector space. The restriction of the complex symplectic form is thus purely imaginary and may be expressed in terms of the generating function of the transformation. We provide explicit formulas; moreover, as an application, we give an explicit general formula for the metaplectic representation of the real symplectic group.
	
\end{abstract}

\mgap

\important{Keywords:} Complex symplectic linear algebra, linear symplectomorphisms, Lagrangian submanifolds, the metaplectic representation

\sgap

\important{Subject Classifications:} 37J10, 51A50, 70H15, 81S10

\mgap

\section{Introduction}\label{S:Intro}

\subsection{Overview}

As part of our symplectic upbringing, our ancestors impressed upon us the Symplectic Creed:
\begin{quote}
	``EVERYTHING IS A LAGRANGIAN SUBMANIFOLD'' \cite{R:Weinstein}.
\end{quote}
Obviously false if taken literally, rather than a ``creed'' it might be called the Maslow-Weinstein hammer, or, in French, \emph{la d\'{e}formation professionnelle symplectique}, saying that ``if all you have is a [symplectic form], everything looks like a [Lagrangian submanifold],'' or, in other words, to a symplectic geometer, everything should be expressed in terms of Lagrangian submanifolds. In this paper we consider a vector space endowed with \emph{two} symplectic forms, namely the real and imaginary parts $\Real\omega^{\Cbb}$ and $\Imag\omega^{\Cbb}$ of a complex symplectic form $\omega^{\Cbb}$, and begin with the simple observation that 
\begin{quote}
	NOT EVERY LAGRANGIAN SUBMANIFOLD [with respect to $\Real\omega^{\Cbb}$] \\
	IS A LAGRANGIAN SUBMANIFOLD [with respect to $\Imag\omega^{\Cbb}$].
\end{quote}
We study its implications for the classification of real linear symplectomorphisms $\calh$, as the graph of $\calh$ is essentially by definition a Lagrangian subspace with respect to $\Real\omega^{\Cbb}$; we ask, with some abuse of language:

\sgap

\important{Open Problem:} Is every $2n\times 2n$ skew-symmetric matrix of the form $\Imag\omega^{\Cbb}|_{\graph\calh}$ for some $\calh$?

\sgap

We believe that an answer would shed some light on the structure of linear symplectomorphisms. While our primary reason for writing this article is to precisely formulate the above open problem, which we do in Section~\ref{S:problem}, our primary technical result is to rewrite it in terms of generating functions; after all, if one guiding principle is the Symplectic Creed, another is that ``Symplectic Topology is the Geometry of Generating Functions'' \cite{R:Viterbo}. Or, to go further back, while Sir William Rowan Hamilton first conceived of generating functions (or as he called them, \emph{characteristic functions}) as mathematical tools in his symplectic formulation of optics, he later found, in his symplectic formulation of classical mechanics, that the generating function for a physical system is the least action function, in a sense that we will not make precise \cite{R:AbeMarsden}, \cite{R:HamDynOne}; this gives a striking connection with the calculus of variations. Moreover, in Fresnel optics and quantum mechanics, the generating function is used as the phase function of an oscillatory integral operator; the integral operator is said to ``quantize'' the corresponding symplectomorphism \cite{R:GrSj}, \cite{R:GuStPhy}. (Loosely speaking, when differentiating the integral, one finds that the phase function must satisfy the Hamilton-Jacobi equation.) This topic will be touched upon in Section~\ref{S:metaplectic}. For us, the generating function corresponding to the linear symplectomorphism $\calh$ is the scalar-valued function $\Phi$ in our main theorem:

\begin{theorem}\label{T:mainthm}
	For each $\calh\in\Sp$ there exists a quadratic form $\Phi:\Cbb_{z}^{n}\times\Rbb_{\theta}^{2n}\to \Rbb$ such that 
	\begin{equation*}
		\graph_{\Cbb}\calh=\{(z,-2\tfrac{\partial\Phi}{\partial z}(z,\theta));\,\,\tfrac{\partial\Phi}{\partial\theta}(z,\theta)=0\},
	\end{equation*}
	and $\omega^{\Cbb}|_{\graph_{\Cbb}\calh}$ is given by
	\begin{equation}
		\begin{aligned}
		\omega^{\Cbb}&((z,-2\tfrac{\partial\Phi}{\partial z}(z,\theta)),(w,-2\tfrac{\partial\Phi}{\partial z}(w,\eta)))\\
		&\qquad=2\sum_{j=1}^{n}\sum_{\ell=1}^{2n}\frac{\partial^{2}\Phi}{\partial z_{j}\partial \theta_{\ell}}(z_{j}\eta_{\ell}-w_{j}\theta_{\ell})
		+2\sum_{j,m=1}^{n}\frac{\partial^{2}\Phi}{\partial z_{j}\partial \overline{z}_{m}}(z_{j}\overline{w}_{m}-w_{j}\overline{z}_{m}).
		\end{aligned}
	\end{equation}
	Moreover, our construction provides an explicit general formula for $\Phi$.
\end{theorem}
Our notation will be explained in the following subsection, along with the necessary background and a restatement of the open problem. We prove the theorem in Section~\ref{S:genf}, and in Section~\ref{S:metaplectic} we show how our construction seems to adequately answer a question of Folland regarding the metaplectic representation \cite{R:FollandHAPS}. We conclude with a broad indication of future work. In the appendix we both (A) give additional linear-algebraic background and a few new elementary results relevant to our problem and (B) give an additional restatement of our open problem.

\mgap

\subsection{Background and restatement of the problem}\label{S:problem}

In a real symplectic vector space there is already a natural complex structure; the model example is $\Rbb^{2n}$ with the $2n\times 2n$ matrix $\Js=\begin{pmatrix}0&-I\\I&0\end{pmatrix}$, where of course $\Js^{2}=-I$. What we mean by ``complex symplectic linear algebra'' is something else; we instead consider $\Cbb^{2n}$ with the above matrix $\Js$, that is, we consider 
$$\omega^{\Cbb}=\sum_{j=1}^{n}d\zeta_{j}\wedge dz_{j}\quad \text{on $\Cbb_{z}^{n}\times\Cbb_{\zeta}^{n}$}$$
(a nondegenerate alternating bilinear form over $\Cbb$). The basic formalism of complex symplectic linear algebra is not new; indeed, complex symplectic structures naturally appear in the theory of differential equations and have been studied through that lens (see, for example, \cite{R:Schapira} and \cite{R:SjSAM}, or \cite{R:EverittMarkus} for another perspective). The point of view of this paper is that elementary linear-algebraic aspects remain unexplored in the complex case and may help us better understand the real case.

\mgap

A \emph{symplectic vector space} over a field\footnote{Duistermaat's book on Fourier integral operators \cite{R:DuistFIO} contains a brief treatment of symplectic vector spaces over a general field.} $K$ is by definition a pair $(V,\omega)$, where $V$ is a finite-dimensional vector space over $K$ and $\omega$ is a nondegenerate alternating bilinear form on $V$. The basic example is $\Rbb_{x}^{n}\times\Rbb^{n}_{\xi}$ with the symplectic form
$\omega=\sum_{j=1}^{n}d\xi_{j}\wedge dx_{j}$:
\begin{equation}\label{E:basicexample}
	\omega((x,\xi),(x',\xi'))=\sum_{j=1}^{n}(\xi_{j}x'_{j}-x_{j}\xi'_{j}).
\end{equation}
In fact, this is essentially the only example: for a general symplectic vector space $(V,\omega)$ over a field $K$ one can find a basis 
$\{e_{1},\ldots,e_{n},f_{1},\ldots,f_{n}\}$ for $V$ such that, for all $j,k$,
\begin{equation*}
	\omega(e_{j},e_{k})=0,\qquad
	\omega(f_{j},f_{k})=0,\quad\text{and}\qquad
	\omega(f_{j},e_{k})=\delta_{jk}.
\end{equation*}
Such a basis is called a \emph{symplectic basis}, and $\omega$ is of the form (\ref{E:basicexample}) in these coordinates. (In particular, a symplectic vector space is necessarily even-dimensional.) Note that $\omega$ vanishes on the span of the $e_{j}$, and it vanishes on the span of the $f_{j}$; such a subspace is called a \emph{Lagrangian subspace}: a maximal subspace on which $\omega$ vanishes. (A Lagrangian subspace of $V$ is necessarily of dimension $n$.)

\mgap

The symplectic formalism is fundamental in Hamiltonian mechanics: the symplectic form provides an isomorphism between tangent space and cotangent space, mapping the Hamiltonian vector field of a function $f$ to the differential of $f$:
$$df=\omega(\cdot, H_{f}).$$

\mgap

A \emph{linear symplectomorphism} $T$ on $(V,\omega)$ is a linear isomorphism on $V$ such that $T^{\ast}\omega=\omega$, that is,
$$\omega(Tv,Tv')=\omega(v,v')\quad\text{for all $v,v'\in V$.}$$
This is equivalent to the property that a symplectic basis is mapped to a symplectic basis.

\mgap

We now let $(V,\omega)$ be a real symplectic vector space. 
Then
$$(V\times V,\omega\oplus-\omega)$$ is a real symplectic vector space. We write
$\omega_{0}=\omega\oplus-\omega$ so that, by definition,
$$\omega_{0}((v,w),(v',w'))=\omega(v,v')-\omega(w,w').$$ 
The following standard fact\footnote{This is a classical result; for a broad perspective, see Terence Tao's blog post \cite{R:TaoClGr}.} justifies this choice of the symplectic form:

\mgap

\begin{quote}
	$\calh:\,V\to V$ is a linear symplectomorphism if and only if 
	$$\graph\calh=\{(v,\calh (v));\,\,v\in V\}$$
	is a Lagrangian subspace of $(V\times V,\omega_{0})$.
\end{quote}

\mgap

For a basic example, let 
\begin{align*}
	\calh:\,\,\Rbb_{x}^{n}\times\Rbb_{\xi}^{n}&\to \Rbb_{y}^{n}\times\Rbb_{\eta}^{n}\\
	(x,\xi)&\mapsto (y,\eta)
\end{align*}
be a linear symplectomorphism. Then $\graph\calh$ is a Lagrangian subspace for 
$$\omega^{\Rbb}=\sum_{j=1}^{n}d\xi_{j}\wedge dx_{j}-d\eta_{j}\wedge dy_{j}.$$

\mgap

The point of view of this paper is to consider $\graph\calh$ as an $\Rbb$-linear subspace of a complex symplectic vector space. After all, with $z_{j}=x_{j}+iy_{j}$ and
$\zeta_{j}=\xi_{j}+i\eta_{j}$, we have the complex symplectic form
$$\omega^{\Cbb}=\sum_{j=1}^{n}d\zeta_{j}\wedge dz_{j}\quad \text{on $\Cbb_{z}^{n}\times\Cbb_{\zeta}^{n}$,}$$
which induces the two \emph{real} symplectic forms 
\begin{align*}
	\Real\omega^{\Cbb}&=\sum_{j=1}^{n}d\xi_{j}\wedge dx_{j}-d\eta_{j}\wedge dy_{j}\quad\text{and}\\
	\Imag\omega^{\Cbb}&=\sum_{j=1}^{n}d\xi_{j}\wedge dy_{j}+d\eta_{j}\wedge dx_{j}\quad\text{on $\Rbb_{x,\xi}^{2n}\times \Rbb_{y,\eta}^{2n}$.}
\end{align*}
We then say that an $\Rbb$-linear $2n$-dimensional subspace of $\Rbb_{x,\xi}^{2n}\times \Rbb_{y,\eta}^{2n}$ is an \emph{$R$-Lagrangian subspace} if it is Lagrangian with respect to $\Real\omega^{\Cbb}$, and an \emph{$I$-Lagrangian subspace} if it is Lagrangian with respect to $\Imag\omega^{\Cbb}$. Thus the graph of $\calh:\,\Rbb_{x,\xi}^{2n}\to \Rbb_{y,\eta}^{2n}$ may be considered as an $R$-Lagrangian subspace of 
$(\Cbb_{z}^{n}\times\Cbb_{\zeta}^{n},\omega^{\Cbb})$.

\mgap

Writing  a symplectic matrix $\calh\in\Sp$ as $\calh=\begin{pmatrix}A&B\\ C&D\end{pmatrix}$, we have
$$\graph\calh=\{((x,\xi),(Ax+B\xi,Cx+D\xi));\,\,(x,\xi)\in \Rbb^{2n}\};$$
or, in terms of $(z,\zeta)$, we have
$$\graph_{\Cbb}\calh=\{(x+i(Ax+B\xi),\xi+i(Cx+D\xi));\,\,(x,\xi)\in \Rbb^{2n}\}.$$
Thus 
$$\omega^{\Cbb}|_{\graph_{\Cbb}\calh}=i\,\Imag\omega^{\Cbb}|_{\graph\calh}$$ 
is given by 
\begin{align*}
	&\omega^{\Cbb}((x+i(Ax+B\xi),\xi+i(Cx+D\xi)),(x'+i(Ax'+B\xi'),\xi'+i(Cx'+D\xi')))\\
	&\qquad\qquad =i\begin{pmatrix}x^{T}&\xi^{T}\end{pmatrix}
	\begin{pmatrix}C^{T}-C&-A^{T}-D\\ A+D^{T}&B-B^{T}\end{pmatrix}
	\begin{pmatrix}x'\\ \xi'\end{pmatrix}.
\end{align*}
The symplectic form $\Real\omega^{\Cbb}$ vanishes, but the symplectic form $\Imag \omega^{\Cbb}$ might \emph{not} vanish; that is, $\graph_{\Cbb}\calh$ is $R$-Lagrangian but not necessarily $I$-Lagrangian.

We have thus defined a map from the group of symplectic matrices to the space of skew-symmetric matrices
\begin{align*}
	\mmap:\,\,\Sp&\to \mathfrak{so}(2n,\Rbb)\\
	\begin{pmatrix}A&B\\C&D\end{pmatrix}&\mapsto 
	\begin{pmatrix}
		C^{T}-C&-A^{T}-D\\
		A+D^{T}&B-B^{T}
	\end{pmatrix}.
\end{align*}

\mgap

We can thus restate our open problem:

\sgap

\important{Open Problem:} Is the map $\mmap:\,\,\Sp\to \mathfrak{so}(2n,\Rbb)$ a surjection? 

\mgap

While we do not solve this problem, the main result of the paper is Theorem~\ref{T:mainthm}; we can explicitly construct a generating function $\Phi$ for $\calh$ and thus give an alternate characterization of $\omega^{\Cbb}|_{\graph_{\Cbb}\calh}$ and hence of $\mmap$.

\mgap\mgap

\section{In terms of generating functions: the proof of the theorem}\label{S:genf}

Generating functions (in the sense of symplectic geometry) were discovered by Sir William Rowan Hamilton in his extensive work on optics. In modern language (and in the linear case), light rays are specified by the following data: $\Rbb^{2}_{x}$ is a plane of initial positions perpendicular to the optical axis of the system, $\xi\in\Rbb^{2}$ are the initial ``directions'' (multiplied by the index of refraction), $\Rbb^{2}_{y}$ is a plane of terminal positions, and $\eta\in\Rbb^{2}$ are the terminal directions. The spaces $\Rbb^{4}_{x,\xi}$ and $\Rbb^{4}_{y,\eta}$ are given the standard symplectic structures. Taken piece by piece, the optical system consists of a sequence of reflections and refractions for each light ray, the laws of which were long known; Hamilton's discovery was that, taken as a whole, the optical system is determined by a single function, the \emph{generating function}, or, as Hamilton called it, the \emph{characteristic function}, of the optical system. The transformation from initial conditions to terminal conditions is a symplectomorphism expressible in terms of a single scalar-valued function, ``by which means optics acquires, as it seems to me, an uniformity and simplicity in which it has been hitherto deficient'' (\cite{R:HamRaysOne}, Section~IV, Paragraph~20).\footnote{There are different types of generating functions in symplectic geometry, and, as Arnold writes, ``[the apparatus of generating functions] is unfortunately noninvariant and it uses, in an essential way, the coordinate structure in phase space'' (\S 47 of \cite{R:Arnold}). For our purposes, we may take the term ``generating function'' to broadly mean a scalar-valued function which generates a symplectomorphism (or, more generally, a Lagrangian submanifold) in the same sense that a potential function generates a conservative vector field. Our generating functions are denoted by the symbol $\Phi$ below.}

The optical framework gives an intuitive reason why, in the symplectic matrix $\calh=\begin{pmatrix}A&B\\C&D\end{pmatrix}$, the rank of $B$ plays a special role in characterizing $\calh$ and thus its generating function. Again, $\calh$ maps the initial (position, ``direction'')-pair $(x,\xi)$ to the terminal (position, ``direction'')-pair 
$$\begin{pmatrix}y\\ \eta \end{pmatrix}=\begin{pmatrix}Ax+B\xi\\ Cx+D\xi\end{pmatrix}.$$
The case $B=0$ corresponds to perfect focusing: all the rays from a given position $x$ arrive at the same position $y$, resulting in a perfect image. And the case $\det B\neq 0$ corresponds to \emph{no} such focusing: two rays with initial position $x$ but different initial ``directions'' must arrive at different positions $y$. (See \cite{R:GuStPhy} for an exposition of symplectic techniques in optics.)

\mgap

\subsection{When $B$ is invertible}\label{S:genfBinv}

We recall that 
\begin{equation*}
	\graph_{\Cbb}\calh
	=\{(x+i(Ax+B\xi),\xi+i(Cx+D\xi));\,\,(x,\xi)\in \Rbb^{2n}\},
\end{equation*}
taken over the reals, is an $R$-Lagrangian subspace of $(\Cbb^{n}_{z}\times\Cbb^{n}_{\zeta},\omega^{\Cbb})$, and we note that
\begin{align*}
	\pi:\,\,\graph_{\Cbb}\calh &\to\Cbb^{n}\\
	(z,\zeta)&\mapsto z
\end{align*}
is an $\Rbb$-linear transformation whose kernel is given by $(x,\xi)\in\{0\}\times \text{ker}B$. Thus it is an $\Rbb$-linear isomorphism if and only if $B$ is invertible. In this case, the general theory of symplectic geometry gives the existence of a real $C^{\infty}$ function $\Phi$ defined on $\graphch$ such that $\graphch = \{z,-2\frac{\partial\Phi}{\partial z}(z); z\in \Cbb^n\}$.

\mgap

Hence if $\det B\neq 0$, then
\begin{align*}
	\graph_{\Cbb}\calh
	&=\{(x+i(Ax+B\xi),\xi+i(Cx+D\xi));\,\,(x,\xi)\in\Rbb^{2n}\}\\
	&=\{(z,-2\tfrac{\partial\Phi}{\partial z}(z));\,\,z\in\Cbb^{n}\}\\
	&=\{(p+iq,B^{-1}(q-Ap)+i(Cp+DB^{-1}(q-Ap)));\,\,p+iq\in\Cbb^{n}\},
\end{align*}
where we write $z=p+iq$, so that
\begin{equation}\label{E:genfpq}
	\Phi(p,q)=\tfrac{1}{2}p^{T}B^{-1}Ap-p^{T}B^{-1}q+\tfrac{1}{2}q^{T}DB^{-1}q.
\end{equation}
This function appears in equation (4.54) of \cite{R:FollandHAPS} and in \S 11 of \cite{R:GuStPhy}. (Note that $B^{-1}A$ and $DB^{-1}$ are symmetric since $\calh$ is symplectic.)
Substituting $p=x$ and $q=Ax+B\xi$, we arrive at the following expression, with the obvious abuse of notation:
\begin{equation}\label{E:genfxxi}
    \Phi(x,\xi)
	=\tfrac{1}{2}x^{T}A^{T}Cx+x^{T}C^{T}B\xi+\tfrac{1}{2}\xi^{T}B^{T}D\xi.
\end{equation}
Or, writing $\Phi$ with respect to $z$ and $\overline{z}$, we have
\begin{align*}
	\Phi(z)&=
	\tfrac{1}{8}z^{T}(B^{-1}A+2iB^{-1}-DB^{-1})z\\
	&\qquad +\tfrac{1}{4}\overline{z}^{T}(B^{-1}A-i(B^{T})^{-1}+iB^{-1}+DB^{-1})z\\
	&\qquad\qquad+\tfrac{1}{8}\overline{z}^{T}(B^{-1}A-2iB^{-1}-DB^{-1})\overline{z}.
\end{align*}
Thus
\begin{equation*}
	\left(\frac{\partial^{2}\Phi}{\partial z_{j}\partial \overline{z}_{k}}\right)
	=\tfrac{1}{4}(B^{-1}A-iB^{-1}+i(B^{T})^{-1}+DB^{-1}).
\end{equation*}

\mgap

We can directly compute $\omega^{\Cbb}$ restricted to $\graph_{\Cbb}\calh$ in terms of $z$ and $\overline{z}$:
\begin{align*}
	\omega^{\Cbb}((z,-2\tfrac{\partial\Phi}{\partial z}(z)),(z',-2\tfrac{\partial\Phi}{\partial z}(z')))
	&=4i\,\Imag\left(\sum_{j,k}z_{j}\left(\frac{\partial^{2}\Phi}{\partial z_{j}\partial\overline{z}_{k}}\right)\overline{z}'_{k}\right)\\
	&=2\sum_{j,k}\frac{\partial^{2}\Phi}{\partial z_{j}\partial\overline{z}_{k}}(z_{j}\overline{z}'_{k}-z'_{j}\overline{z}_{k})
\end{align*}
If we substitute
\begin{align*}
	z&=x+i(Ax+B\xi)\\
	z'&=x'+i(Ax'+B\xi')
\end{align*}
then after a lengthy mechanical calculation we recover the expression
\begin{align*}
	\omega^{\Cbb}((z,-2\tfrac{\partial\Phi}{\partial z}(z)),(z',-2\tfrac{\partial\Phi}{\partial z}(z'))) = 2\sum_{j,k}\frac{\partial^{2}\Phi}{\partial z_{j}\partial\overline{z}_{k}}(z_{j}\overline{z}'_{k}-z'_{j}\overline{z}_{k})=i\begin{pmatrix}x^{T}& \xi^{T}\end{pmatrix}\mmap(\calh)\begin{pmatrix}x'\\ \xi'\end{pmatrix}.
\end{align*}

\mgap\mgap

\subsection{When $B$ is not invertible}\label{S:genfBdeg}

When $B$ is \emph{not} invertible, we seek $\Phi=\Phi(z,\theta)\in C^{\infty}(\Cbb^{n}\times\Rbb^{N})$ such that
\begin{equation}\label{E:cxgraphinit}
	\graph_{\Cbb}\calh
	=\{(z,-2\tfrac{\partial\Phi}{\partial z}(z,\theta));\,\,\tfrac{\partial\Phi}{\partial\theta}(z,\theta)=0\}.
\end{equation}
We follow the general method outlined by Guillemin and Sternberg \cite{R:GuStGA}.

\mgap

Let 
\begin{align*}
	W&=\graph_{\Cbb}\calh\\
	X&=\{(z,0);\,\,z\in\Cbb^{n}\}\\
	Y&=\{(0,\zeta);\,\,\zeta\in\Cbb^{n}\}.
\end{align*}
Since $W$ is an $R$-Lagrangian subspace, we know that $W\cap Y$ and $PW\subset X$ are orthogonal with respect to 
$\Real\omega^{\Cbb}$, where $P$ is the projection onto $X$ along $Y$. Indeed,
\begin{align*}
	W\cap Y&=\{(0,\xi+iD\xi);\,\,\xi\in\text{ker}B\},\\
	PW&=\{(x+i(Ax+B\xi),0);\,\,(x,\xi)\in\Rbb^{2n}\},
\end{align*}
and we can check directly that, with $\xi\in\text{ker}B$,
\begin{equation*}
	\omega^{\Cbb}((0,\xi+iD\xi),(x'+i(Ax'+B\xi'),0))
	=i\left[\xi^{T}(A+D^{T})x'+\xi^{T}B\xi'\right].
\end{equation*}

\mgap

Since $\graph_{\Cbb}\calh$ is not a $\Cbb$-linear subspace but an $\Rbb$-linear subspace, for now we prefer to write
\begin{align*}
	W\cap Y&=\{(0,\xi;0,D\xi);\,\,\xi\in\text{ker}B\}\\
	PW&=\{(x,0;Ax+B\xi,0);\,\,(x,\xi)\in\Rbb^{2n}\}.
\end{align*}
We note that $PW\oplus (W\cap Y)$ has real dimension $2n$, hence is a Lagrangian subspace of $(\Rbb^{4n},\Real\omega^{\Cbb})$. 

\mgap

We seek to write $\graph\calh$ as the graph of a function from $PW\oplus (W\cap Y)$ to a complementary Lagrangian subspace; as a first step, we choose a convenient symplectic basis. We let $\{b_{1},\ldots,b_{k}\}$ be an orthonormal basis for $\ker B$ and extend to an orthonormal basis $\{b_{1},\ldots,b_{n}\}$ for $\Rbb^{n}$, so that
$$\{(0,b_{j};0,Db_{j});\,\,j=1,\ldots,k\}\quad\text{is a basis for $W\cap Y$, and}$$
$$\{(0,0;Bb_{j},0);\,\,j=k+1,\ldots,n\}\cup \{(b_{j},0;Ab_{j},0);\,\,j=1,\ldots,n\}\quad\text{is a basis for $PW$.}$$ We then extend to the following symplectic basis for $(\Rbb^{4n},\Real\omega^{\Cbb})$: 
\begin{equation}\label{E:symplbasis}
\begin{aligned}
	\{(0,0;Ab_{j},0);\,\,j=1,\ldots,k\}&\leftrightarrow \{(0,b_{j};0,Db_{j});\,\,j=1,\ldots,k\}\\
	\{(0,0;Bb_{j},0);\,\,j=k+1,\ldots,n\}&\leftrightarrow \{(0,A^{T}\beta_{j};0,\beta_{j});\,\,j=k+1,\ldots,n\}\\
	\{(b_{j},0;Ab_{j},0);\,\,j=1,\ldots,n\}&\leftrightarrow \{(0,-b_{j};0,0);\,\,j=1,\ldots,n\},
\end{aligned}
\end{equation}
where the $\{\beta_{j}\}_{j=k+1}^{n}$ satisfy
\begin{equation}\label{E:betaconds}
	\begin{cases}
		A^{T}\beta_{j}\in(\ker B)^{\perp}=\Imag B^{T}\\
		b_{J}\cdot B^{T}\beta_{j}=\delta_{Jj}\quad\text{for all $J\in\{k+1,\ldots,n\}$}.
	\end{cases}
\end{equation}
One advantage of using this particular symplectic basis (\ref{E:symplbasis}) is that the vectors on the left are all ``horizontal,'' and the vectors on the right are all ``vertical.'' (The arrows signify the symplectically-dual pairs.) 

\mgap

The following proposition implies the existence of $\{\beta_j\}_{j=k+1}^n$.

\mgap

\begin{prop}
	The set $\{Ab_{1},\ldots,Ab_{k},Bb_{k+1},\ldots,Bb_{n}\}$ is a basis for $\Rbb^{n}$.
\end{prop}
\begin{proof}
	Suppose
	$$\sum_{j=1}^{k}\alpha_{j}Ab_{j}+\sum_{j=k+1}^{n}\alpha_{j}Bb_{j}=0.$$
	We take the dot product with $Db_{J}$, $J\in\{1,\ldots,k\}$, to get $\alpha_{1}=\cdots=\alpha_{k}=0$, and the rest are zero by the linear independence of $\{Bb_{k+1},\ldots,Bb_{n}\}$.
\end{proof}

\mgap

Thus for $J\in\{k+1,\ldots,n\}$ we can take $\beta_{J}$ to be the unique vector orthogonal to the set
$$\{Ab_{1},\ldots,Ab_{k},Bb_{k+1},\ldots,\widehat{Bb_{J}},\ldots,Bb_{n}\}$$
(where the wide hat denotes omission) and satisfying
$$\beta_{J}\cdot Bb_{J}=1.$$

\mgap

We will now describe $\graph\calh$ in terms of the above symplectic coordinate system: we write a general linear combination of the $4n$ vectors and find necessary and sufficient conditions on the coefficients to make the vector in $\graph\calh$. Explicitly, we write the general vector in $\Rbb^{4n}$ as
\begin{equation}\label{E:genVector}
\begin{aligned}
	&\sum_{j=1}^{k}t_{j}'(0,0;Ab_{j},0)+\sum_{j=k+1}^{n}t_{j}''(0,0;Bb_{j},0)+\sum_{j=1}^{n}t_{n+j}''(b_{j},0;Ab_{j},0)\\
	&\qquad+\sum_{j=1}^{k}\theta_{j}'(0,b_{j};0,Db_{j})+\sum_{j=k+1}^{n}\theta_{j}''(0,A^{T}\beta_{j};0,\beta_{j})
	+\sum_{j=1}^{n}\theta_{n+j}''(0,-b_{j};0,0).
\end{aligned}
\end{equation}
(the superscript primes and double-primes are not necessary but are useful for bookkeeping), and we will describe $\graph\calh$ as $(t',\theta'')$ as a function of $(t'',\theta')$.

\mgap

We have the following necessary and sufficient conditions for the vector (\ref{E:genVector}) to be in $\graph\calh$:
\begin{equation*}
	\begin{aligned}
		\sum_{j=1}^{k}t_{j}'Ab_{j}-\sum_{j=k+1}^{n}\theta_{j}''AB^{T}\beta_{j}+\sum_{j=k+1}^{n}\theta_{n+j}''Bb_{j}&=-\sum_{j=k+1}^{n}t_{j}''Bb_{j}\\
		-\sum_{j=k+1}^{n}\theta_{j}''CB^{T}\beta_{j}+\sum_{j=1}^{n}\theta_{n+j}''Db_{j}&=\sum_{j=1}^{n}t_{n+j}''Cb_{j}.
	\end{aligned}
\end{equation*}
In matrix form, this says:
\begin{equation}\label{E:bigmatrixeqn}
\begin{aligned}
    &\begin{pmatrix}
		|&\,&|&|&\,&|&|&\,&|\\
		Ab_{1}&\cdots&Ab_{k}&(-AB^{T}\beta_{k+1})&\cdots&(-AB^{T}\beta_{n})&Bb_{1}&\cdots&Bb_{n}\\
		|&\,&|&|&\,&|&|&\,&|\\
		\,&\,&\,&|&\,&|&|&\,&|\\
		\,&0_{n,k}&\,&(-CB^{T}\beta_{k+1})&\cdots&(-CB^{T}\beta_{n})&Db_{1}&\cdots&Db_{n}\\
		\,&\,&\,&|&\,&|&|&\,&|
	\end{pmatrix}
	\begin{pmatrix}
		\,\\
		t'\\
		\,\\
		\,\\
		\theta''\\
		\,
	\end{pmatrix}\\
	&\qquad\qquad=
	\begin{pmatrix}
		|&\,&|&\,&\,&\,\\
		(-Bb_{k+1})&\cdots&(-Bb_{n})&\,&0_{n,n}&\,\\
		|&\,&|&\,&\,&\,\\
		\,&\,&\,&|&\,&|\\
		\,&0_{n,n-k}&\,&Cb_{1}&\cdots&Cb_{n}\\
		\,&\,&\,&|&\,&|
	\end{pmatrix}
	\begin{pmatrix}
		\,\\
		\,\\
		t''\\
		\,\\
		\,
	\end{pmatrix}.
\end{aligned}
\end{equation}
We would now like to invert the matrix on the left to get $(t',\theta'')$ as a function of $(t'',\theta')$. Once we do that, we are close to our goal of expressing $\graph\calh$ in terms of a generating function $\Phi$.


Letting $\Pi$ denote the orthogonal projection onto $\ker B$, we find that the inverse of the matrix on the left side of equation (\ref{E:bigmatrixeqn}) is
\begin{equation*}
	\begin{pmatrix}
		\hdash&Db_{1}&\hdash&\,&\,&\,\\
		\,&\vdots&\,&\,&0_{k,n}&\,\\
		\hdash&Db_{k}&\hdash&\,&\,&\,\\
		\hdash&D(\Pi C^{T}B-I)b_{k+1}&\hdash&\hdash&Bb_{k+1}&\hdash\\
		\,&\vdots&\,&\,&\vdots&\,\\
		\hdash&D(\Pi C^{T}B-I)b_{n}&\hdash&\hdash&Bb_{n}&\hdash\\
		\hdash&(D\Pi A^{T}-I)Cb_{1}&\hdash&\hdash&Ab_{1}&\hdash\\
		\,&\vdots&\,&\,&\vdots&\,\\
		\hdash&(D\Pi A^{T}-I)Cb_{n}&\hdash&\hdash&Ab_{n}&\hdash
	\end{pmatrix}.
\end{equation*}
Thus, defining the functions 
\begin{align*}
	f''_{i}(t'')&=\sum_{j=k+1}^{n}[Bb_{i}\cdot Db_{j}]t_{j}''+\sum_{j=1}^{n}[Bb_{i}\cdot Cb_{j}]t_{n+j}''\quad \text{for $i=k+1,\ldots,n$},\\
	f''_{n+i}(t'')&=\sum_{j=k+1}^{n}[Cb_{i}\cdot Bb_{j}]t_{j}''+\sum_{j=1}^{n}[Ab_{i}\cdot Cb_{j}]t_{n+j}''\quad\, \text{for $i=1,\ldots,n$},
\end{align*}
the equation (\ref{E:bigmatrixeqn}) is equivalent to the conditions
\begin{equation*}
		t'=0,\qquad \theta''=f''(t''),
\end{equation*}
and we note that
$$\frac{\partial f_{i}''}{\partial t_{j}''}=\frac{\partial f_{j}''}{\partial t_{i}''} \qquad\text{for all $i,j\in{k+1,\ldots,n}$,}$$
allowing us to define 
\begin{align*}
	F(t'')
	&=\frac{1}{2}\sum_{i=k+1}^{n}\sum_{j=k+1}^{n}t_{i}''[Bb_{i}\cdot Db_{j}]t_{j}''\\
	&\qquad\qquad+\sum_{i=k+1}^{n}\sum_{j=1}^{n}t_{i}''[Bb_{i}\cdot Cb_{j}]t_{n+j}''
	+\frac{1}{2}\sum_{i=1}^{n}\sum_{j=1}^{n}t_{n+i}''[Ab_{i}\cdot Cb_{j}]t_{n+j}''.
\end{align*}
Thus the conditions for the vector to be in $\graph\calh$ are equivalent to the conditions
\begin{equation*}
	t'=0,\qquad\qquad\frac{\partial F}{\partial t''}(t'')=\theta''.
\end{equation*}

\mgap

We now define
\begin{equation*}
	\vp(t',t'';\theta',\theta'')=\theta'\cdot t'+F(t'')+(\theta''-f''(t''))^{2}.
\end{equation*}
Then in $(t',t'';\theta',\theta'')$-coordinates, $\graph\calh$ is given as
\begin{equation*}
	\left\{\left(t',t'';\frac{\partial\vp}{\partial t'},\frac{\partial\vp}{\partial t''}\right);\,\,\frac{\partial\vp}{\partial\theta'}=0,\,\,\frac{\partial\vp}{\partial\theta''}=0\right\}.
\end{equation*}
Or, written in terms of the standard basis, $\graph\calh$ is the set of
\begin{equation}\label{E:graphPhi}
\begin{aligned}
	&\sum_{j=1}^{k}t_{j}'(0,0;Ab_{j},0)+\sum_{j=k+1}^{n}t_{j}''(0,0;Bb_{j},0)+\sum_{j=1}^{n}t_{n+j}''(b_{j},0;Ab_{j},0)\\
	&\qquad +\sum_{j=1}^{k}\frac{\partial\vp}{\partial t_{j}'}(t,\theta)(0,b_{j};0,Db_{j})
	+\sum_{j=k+1}^{n}\frac{\partial\vp}{\partial t_{j}''}(t,\theta)(0,A^{T}\beta_{j};0,\beta_{j})\\
	&\qquad\qquad\qquad\qquad\qquad\qquad+\sum_{j=1}^{n}\frac{\partial\vp}{\partial t_{n+j}''}(t,\theta)(0,-b_{j};0,0)
\end{aligned}
\end{equation}
such that
$$\frac{\partial\vp}{\partial\theta}(t,\theta)=0.$$

\mgap

We return to complex coordinates, in the standard basis; for that purpose we write the ``horizontal'' parts of (\ref{E:graphPhi}) as:
\begin{align*}
	z:=\sum_{j=1}^{k}it_{j}'Ab_{j}+\sum_{j=k+1}^{n}it_{j}''Bb_{j}+\sum_{j=1}^{n}t_{n+j}''(I+iA)b_{j}.
\end{align*}
That is,
\begin{align*}
	\Real z&=\sum_{j=1}^{n}t_{n+j}''b_{j}\\
	\Imag z&=\sum_{j=1}^{k}t_{j}'Ab_{j}+\sum_{j=k+1}^{n}t_{j}''Bb_{j}+\sum_{j=1}^{n}t_{n+j}''Ab_{j}.
\end{align*}
With the same notation as before, the inverse transformation is given by
\begin{equation}\label{E:xfnofz}
\begin{aligned}
	t_{j}'&=-b_{j}\cdot\Real z+Db_{j}\cdot \Imag z \qquad\quad\text{for $j\in\{1,\ldots,k\}$,}\\
	t_{j}''&=-A^{T}\beta_{j}\cdot\Real z+\beta_{j}\cdot \Imag z\qquad\,\text{for $j\in\{k+1,\ldots,n\}$, and}\\
	t_{n+j}''&=b_{j}\cdot \Real z\qquad\qquad\qquad\qquad\quad\,\,\text{for $j\in\{1,\ldots,n\}$.}
\end{aligned}
\end{equation}

\mgap

We write the ``vertical'' part of (\ref{E:graphPhi}) as:
\begin{equation}\label{E:zetavert}
\begin{aligned}
	\Real\zeta&=\sum_{j=1}^{k}\frac{\partial\vp}{\partial t_{j}'}(t,\theta)b_{j}+\sum_{j=k+1}^{n}\frac{\partial\vp}{\partial t_{j}''}(t,\theta)A^{T}\beta_{j}
		-\sum_{j=1}^{n}\frac{\partial\vp}{\partial t_{n+j}''}(t,\theta)b_{j}\\
	\Imag\zeta&=\sum_{j=1}^{k}\frac{\partial\vp}{\partial t_{j}'}(t,\theta)Db_{j}+\sum_{j=k+1}^{n}\frac{\partial\vp}{\partial t_{j}''}(t,\theta)\beta_{j}.
\end{aligned}
\end{equation}
Using $t=t(z)$ to denote the transformation (\ref{E:xfnofz}), we define 
$$\Phi(z,\theta):=\vp(t(z),\theta)$$
so that (\ref{E:zetavert}) says
$$\zeta=-2\tfrac{\partial\Phi}{\partial z}(z,\theta).$$
In summary, we now have the following expression for $\graph_{\Cbb}\calh$:
\begin{equation}\label{E:finalcxgraph}
	\graph_{\Cbb}\calh=\{(z,-2\tfrac{\partial\Phi}{\partial z}(z,\theta));\,\,\tfrac{\partial\Phi}{\partial \theta}(z,\theta)=0\},
\end{equation}
where the $\theta\in\Rbb^{2n}$ are considered as auxiliary parameters, as in (\ref{E:cxgraphinit}).

\mgap

As for $\omega^{\Cbb}|_{\graph_{\Cbb}\calh}$, we use the expression
\begin{equation*}
	\frac{\partial\Phi}{\partial z}(z,\theta)
	=\frac{\partial^{2}\Phi}{\partial z\partial\theta}\cdot\theta+ \frac{\partial^{2}\Phi}{\partial z^{2}}\cdot z
	+\frac{\partial^{2}\Phi}{\partial z\partial\overline{z}}\cdot\overline{z}
\end{equation*}
to compute
\begin{equation}\label{E:zomegacxgr}
\begin{aligned}
	&\omega^{\Cbb}((z,-2\tfrac{\partial\Phi}{\partial z}(z,\theta)),(w,-2\tfrac{\partial\Phi}{\partial z}(w,\eta)))\\
	&\qquad=2z\cdot\tfrac{\partial\Phi}{\partial z}(w,\eta)-2w\cdot\tfrac{\partial\Phi}{\partial z}(z,\theta)\\
	&\qquad=2\sum_{j=1}^{n}\sum_{\ell=1}^{2n}\frac{\partial^{2}\Phi}{\partial z_{j}\partial \theta_{\ell}}(z_{j}\eta_{\ell}-w_{j}\theta_{\ell})
	+2\sum_{j,m=1}^{n}\frac{\partial^{2}\Phi}{\partial z_{j}\partial \overline{z}_{m}}(z_{j}\overline{w}_{m}-w_{j}\overline{z}_{m}),
\end{aligned}
\end{equation}
where the variables are related by the conditions
\begin{equation*}
	\frac{\partial\Phi}{\partial\theta}(z,\theta)=0\qquad\text{and}\qquad \frac{\partial\Phi}{\partial\theta}(w,\eta)=0.
\end{equation*}
Of course, from Section~\ref{S:Intro}, we know that (\ref{E:zomegacxgr}) is equal to
\begin{equation}\label{E:realomegacxgr}
	i\begin{pmatrix}x^{T}&\xi^{T}\end{pmatrix}\mmap(\calh)\begin{pmatrix}x'\\ \xi'\end{pmatrix},
\end{equation}
where
\begin{align*}
	z&=x+i(Ax+B\xi),\\
	w&=x'+i(Ax'+B\xi'),\\
	-2\frac{\partial\Phi}{\partial z}(z,\theta)&=\xi+i(Cx+D\xi),\quad\text{and}\\
	-2\frac{\partial\Phi}{\partial z}(w,\eta)&=\xi'+i(Cx'+D\xi').
\end{align*}
This completes the proof of the theorem.

\sgap

We leave it as an illustrative exercise for the reader to compute $\Phi$ and its derivatives in the special cases when $B=0$ and when $B$ is invertible (to be compared with the generating function (\ref{E:genfpq}) in Section~\ref{S:genfBinv}).

\mgap
\mgap

\section{Application: the Metaplectic Representation}\label{S:metaplectic}

In the previous section, we showed how to associate to a linear symplectomorphism $\calh$ a (real-valued) generating function $\Phi$. For the purposes of Fresnel optics and quantum mechanics one then associates to the generating function $\Phi$ an oscillatory integral operator 
\begin{equation}\label{E:mumeta}
	\begin{aligned}
	\mu(\calh): \quad \mathcal{S}(\Rbb^{n})&\to\mathcal{S}'(\Rbb^{n})\\
	u&\mapsto a\, h^{-3n/2}\iint e^{i\Phi(x+iy,\theta)/h}u(x)\,dx\,d\theta.
	\end{aligned}
\end{equation}
The map $\mu:\,\,\calh\to\mu(\calh)$ is called the \emph{metaplectic representation} of the symplectic group, and $\mu(\calh)$ is said to be the ``quantization'' of the classical object $\calh$. As defined, the operator $\mu(\calh)$ maps Schwartz functions to tempered distributions, but in fact it extends to a bounded operator on $L^{2}(\Rbb^{n})$; we choose $a$ so that $\mu(\calh)$ is unitary on $L^{2}(\Rbb^{n})$, and here $0<h$ is a small parameter. These are the operators of ``Fresnel optics,'' a relatively simple model theory for optics which accounts for interference and diffraction, describing the propagation of light of wavelength $h$ \cite{R:GuStPhy}. For the analytic details we refer the reader to a text in semiclassical analysis \cite{R:DimSj}; here we only show that the standard conditions are indeed satisfied.

\mgap

The above (real-valued) generating function $\Phi$, for an arbitrary $\calh\in\Sp$, has the property that the $1$-forms
$d\frac{\partial\Phi}{\partial\theta_{1}},\ldots,d\frac{\partial\Phi}{\partial\theta_{2n}}$
are linearly independent. Equivalently, with the notation from the previous section, the matrix
\begin{align*}
	\begin{pmatrix}
		\frac{\partial^{2}\Phi}{\partial(\Real z)\partial\theta'}&\frac{\partial^{2}\Phi}{\partial(\Real z)\partial\theta''}\\
		\\
		\frac{\partial^{2}\Phi}{\partial(\Imag z)\partial\theta'}&\frac{\partial^{2}\Phi}{\partial(\Imag z)\partial\theta''}\\
		\\
		\frac{\partial^{2}\Phi}{\partial\theta'^{2}}&\frac{\partial^{2}\Phi}{\partial\theta'\partial\theta''}\\
		\\
		\frac{\partial^{2}\Phi}{\partial\theta''\partial\theta'}&\frac{\partial^{2}\Phi}{\partial\theta''^{2}}
	\end{pmatrix}
	=
	\begin{pmatrix}
		|&\,&|&\,\\
		(-b_{1})&\cdots&(-b_{k})&\ast\\
		|&\,&|&\,\\
		|&\,&|&\,\\
		Db_{1}&\cdots&Db_{k}&\ast\\
		|&\,&|&\,\\
		\,&0_{k,k}&\,&0_{k,(2n-k)}\\
		\,&0_{(2n-k),k}&\,&2I_{(2n-k),(2n-k)}
	\end{pmatrix}
\end{align*}
has linearly independent columns. (The asterisks denote irrelevant components.) This condition says that that quadratic form $\Phi=\Phi(z,\theta)$ is a \emph{nondegenerate phase function} in the sense of semiclassical analysis \cite{R:DimSj}. 

\mgap

Folland writes: ``it seems to be a fact of life that there is no simple description of the operator $\mu(\mathcal{A})$ that is valid for all $\mathcal{A}\in \text{Sp}$'' (\cite{R:FollandHAPS}, page 193); however, we believe that (\ref{E:mumeta}), combined with our construction of $\Phi$ in the proof of Theorem~\ref{T:mainthm}, is such a description.

\mgap\mgap

\section{Conclusion}

The open problem and results presented in this paper were motivated by the basic question of the relationship between real and complex symplectic linear algebra. Our approach to this question was to consider a real symplectomorphism as a Lagrangian submanifold with regard to the real part of a complex symplectic form. We believe the resulting problem of the nature of the restriction of the imaginary part of the complex symplectic form to this submanifold (formally, $\mmap(\calh)$ for a symplectomorphism $\calh$) is relevant to the structure of the real symplectic group. (We direct the reader to the Appendix for a list of properties of $\mmap$ and reformulations of our open problem which lend credence to this belief.) Accordingly, we view the main result of this paper as primarily a means for further investigation of the open problem of the image of $\mmap$. In addition to solving our open problem, we believe that, in line with our generating function formulation, it would be interesting to have a ``complexified'' theory of the calculus of variations. At present we only have trivial extensions of the real theory.

\newpage

\appendix
\section*{Appendix}
\renewcommand{\thesubsection}{\Alph{subsection}}

\subsection{Elementary properties of $\mmap$}\label{S:elemprops}

We first note some standard facts about symplectic matrices that are used throughout the paper; for further information, see, for example, \cite{R:Cannas} or \cite{R:FollandHAPS}. We write $$\Js=\begin{pmatrix}0&-I\\I&0\end{pmatrix}$$
for the matrix representing the standard symplectic form.

\mgap

\begin{prop} \cite{R:FollandHAPS}
	Let $\calh\in \text{GL}(2n,\Rbb)$. Then the following are equivalent:
	\begin{enumerate}
		\item $\calh\in\Sp$.
		\item $\calh^{T}\Js\calh=\Js$.
		\item $\calh^{-1}=\Js\calh^{T}\Js^{-1}=\begin{pmatrix}D^{T}&-B^{T}\\-C^{T}&A^{T}\end{pmatrix}$.
		\item $\calh^{T}\in\Sp$.
		\item $A^{T}D-C^{T}B=I$, $A^{T}C=C^{T}A$, and $B^{T}D=D^{T}B$.
		\item $AD^{T}-BC^{T}=I$, $AB^{T}=BA^{T}$, and $CD^{T}=DC^{T}$.
	\end{enumerate}
\end{prop}

\mgap

While $\mmap$ may be extended to all of $\mathbb{M}^{2n}(\Rbb)$,
\begin{equation}\label{E:mmapSp}
	\begin{aligned}
	\mmap:\quad \mathbb{M}^{2n}(\Rbb)&\to \mathfrak{so}(2n,\Rbb)\\
	M&\mapsto\Js M+M^{T}\Js,
	\end{aligned}
\end{equation}
for purposes of our open problem the resulting linearity of $\mmap$ does not seem to help when $\mmap$ is restricted to $\Sp$.

\mgap

The following proposition presents some of the most interesting elementary linear algebraic properties of $\mmap$.

\sgap

\begin{prop}
Let $\mmap\colon \mathbb{M}^{2n}(\Rbb)\to\mathfrak{so}(2n,\Rbb)$ be defined as above. Then:
\begin{enumerate}
\item $\text{ker}(\mmap)=\mathfrak{sp}(2n,\Rbb)$, the symplectic Lie algebra.
\item For any $\calh\in\Sp$, $\mmap(\calh) = \Js(\calh + \calh^{-1})$.

In particular, for $\mathcal{U}\in U(n)=\{\left(\begin{smallmatrix}A&-B\\B&A\end{smallmatrix}\right)\in \Sp\}$ we have $\mathcal{U}^{-1}=\mathcal{U}^{T}$, so $\mmap(\mathcal{U})=\Js(\mathcal{U}+\mathcal{U}^{T})$.
\item  For any $\calh\in\Sp$, $\mmap(\calh)$ is invertible (equivalently, $\Imag\omega^{\Cbb}|_{\graph\calh}$ is nondegenerate) if and only if $-1$ is not a member of the spectrum of $\calh^2$. 
\item For $\calh, \mathcal{R}\in \Sp$, we have $\calh^{T}\mmap(\mathcal{R})\calh=\mmap(\calh^{-1}\mathcal{R}\calh)$.
\end{enumerate}
\end{prop}
The above properties follow immediately from the definition of $\mmap$.

\mgap 

We now take some examples.

\sgap

\important{Examples of Symplectic Matrices and Their Images Under $\mmap$:}
\begin{enumerate}
	\item $$\begin{pmatrix}A&0\\0&(A^{T})^{-1}\end{pmatrix}\mapsto \begin{pmatrix}0&-A^{T}-(A^{T})^{-1}\\A+A^{-1}&0\end{pmatrix}.$$
	In particular, $$\begin{pmatrix}I&0\\0&I\end{pmatrix}\mapsto \begin{pmatrix}0&-2I \\2I&0\end{pmatrix}=2\Js.$$
	
	\mgap
	
	\item For $B=B^{T}$, $\begin{pmatrix}I&B\\0&I\end{pmatrix}\mapsto \begin{pmatrix}0&-2I\\2I&0\end{pmatrix}$.
	
	\mgap
	
	\item For $C=C^{T}$, $\begin{pmatrix}I&0\\C&I\end{pmatrix}\mapsto \begin{pmatrix}0&-2I\\2I&0\end{pmatrix}$.
	
	\mgap
	
	\item $\Js=\begin{pmatrix}0&-I\\I&0\end{pmatrix}\mapsto \begin{pmatrix}0&0\\0&0\end{pmatrix}$.
	
	\mgap
	
	\item For $t\in\Rbb$,
	\begin{equation*}
		\begin{pmatrix}(\cos t)I&(-\sin t)I\\ (\sin t)I&(\cos t)I\end{pmatrix}
		\mapsto \begin{pmatrix}0&-2(\cos t)I\\2(\cos t)I&0\end{pmatrix}.
	\end{equation*}
	
	\item For any $\calh\in\Sp$, we have $\mmap(\calh)=\mmap(\calh^{-1})$.
	
\end{enumerate}

\mgap 

Thus in Examples 2 and 3, $\graph_{\Cbb}\calh$ is an $RI$-subspace ($R$-Lagrangian and $I$-symplectic). 
And in Example 4, $\graph_{\Cbb}\calh$ is a $C$-Lagrangian subspace ($R$-Lagrangian and $I$-Lagrangian).

\mgap

The exact nature of the image of $\mmap$ is an open question. The following is a partial result; for additional partial results, we refer to the appendix.

\mgap

\begin{prop}
	For each $k\in\{0,1,\ldots,n\}$, there exists $\calh_{k}\in\Sp$ such that $\text{rank}(\mmap(\calh_{k}))=2k$. Moreover, for any $\calh\in\Sp$, 
	we have $\ker\mmap(\calh)=\ker(\calh^{2}+I)$.
\end{prop}
\begin{proof}
We fix $k\in\{0,1,\ldots,n\}$ and write
$$(x,\xi)=(x',x'',\xi',\xi''),\quad x',\xi'\in\Rbb^{k},\quad x'',\xi''\in\Rbb^{n-k}.$$
Let 
$$\calh_{k}(x',x'',\xi',\xi'')=(x',-\xi'',\xi',x'').$$
The matrix representation of $\calh_{k}$ is
\begin{equation*}
	\begin{pmatrix}
		I_{k}&\,&0_{k}&\,\\
		\,&0_{n-k}&\,&-I_{n-k}\\
		0_{k}&\,&I_{k}&\,\\
		\,&I_{n-k}&\,&0_{n-k}
	\end{pmatrix}
	\in\Sp.
\end{equation*}
Then
\begin{equation*}
	\mmap(\calh_{k})
	=
	\begin{pmatrix}
		\,&\,&-2I_{k}&\,\\
		\,&\,&\,&0_{n-k}\\
		2I_{k}&\,&\,&\,\\
		\,&0_{n-k}&\,&\,
	\end{pmatrix},
\end{equation*}
so that 
$$\text{rank}(\mmap(\calh_{k}))=2k.$$

The last statement of the proposition follows from (\ref{E:mmapSp}).
\end{proof}

\mgap

\subsection{Restatement of the problem}

It is sometimes convenient to work with the extension of $\mmap$ to all of $\mathbb{M}^{2n}(\Rbb)$:
\begin{equation*}
	\mmap(M)=\Js M+M^{T}\Js.
\end{equation*}
Then $\mmap:\mathbb{M}(2n,\Rbb)\to\mathfrak{so}(2n,\Rbb)$ is a linear epimorphism with kernel $\mathfrak{sp}(2n,\Rbb)$, the symplectic Lie algebra (see, for example, Proposition~4.2 of \cite{R:FollandHAPS}).
Thus the map $\mmap|_{\Sp}$ is surjective if and only if every element of the quotient space $\mathbb{M}(2n,\Rbb)/\mathfrak{sp}(2n,\Rbb)$ contains a symplectic matrix. So our question is:

\mgap

\important{Question:} Can every $M\in\mathbb{M}(2n,\Rbb)$ be written as $M=\calh+\mathcal{A}$ for some $\calh\in\Sp$ and some $\mathcal{A}\in\mathfrak{sp}(2n,\Rbb)$?

\mgap

\begin{prop}
	Every $M\in\mathbb{M}(2n,\Rbb)/\mathfrak{sp}(2n,\Rbb)$ has a unique representative of the form
	$\begin{pmatrix}0&\mathcal{S}_{2}\\ \mathcal{S}_{3}&\mathcal{D}\end{pmatrix}$, where $\mathcal{S}_{2}$ and $\mathcal{S}_{3}$ are skew-symmetric.
\end{prop}
\begin{proof}
	Existence: let $M=\begin{pmatrix}M_{1}&M_{2}\\M_{3}&M_{4}\end{pmatrix}\in\mathbb{M}(2n,\Rbb)$. Since
	$\begin{pmatrix}\alpha&\beta\\ \gamma&\delta\end{pmatrix}\in \mathfrak{sp}(2n,\Rbb)$ if and only if $\delta=-\alpha^{T}$, $\beta=\beta^{T}$, $\gamma=\gamma^{T}$, we may replace $M$ by
	\begin{align*}
		\tilde{M}
		&=M-
		\begin{pmatrix}
			M_{1}&\frac{1}{2}(M_{2}+M_{2}^{T})\\
			\frac{1}{2}(M_{3}+M_{3}^{T})&-M_{1}^{T}
			\end{pmatrix}
			=\begin{pmatrix}
				0&\frac{1}{2}(M_{2}-M_{2}^{T})\\
				\frac{1}{2}(M_{3}-M_{3}^{T})&M_{4}+M_{1}^{T}
			\end{pmatrix}.
	\end{align*}
	
	Uniqueness: Suppose
	\begin{equation*}
		\begin{pmatrix}
			0&\mathcal{S}_{2}\\
			\mathcal{S}_{3}&\mathcal{D}
		\end{pmatrix}
		=
		\begin{pmatrix}
			0&\mathcal{S}'_{2}\\
			\mathcal{S}'_{3}&\mathcal{D}'
		\end{pmatrix}
		\in\mathbb{M}(2n,\Rbb)/\mathfrak{sp}(2n,\Rbb),
	\end{equation*}
	with the $\mathcal{S}_{j}$ and $\mathcal{S}_{j}'$ skew-symmetric. Thus
	\begin{equation*}
		\begin{pmatrix}
			0&\mathcal{S}_{2}-\mathcal{S}_{2}'\\
			\mathcal{S}_{3}-\mathcal{S}_{3}'&\mathcal{D}-\mathcal{D}'
		\end{pmatrix}
		=
		\begin{pmatrix}
			\alpha&\beta\\
			\gamma&-\alpha^{T}
		\end{pmatrix}
		\in\mathfrak{sp}(2n,\Rbb).
	\end{equation*}
	This shows that $\mathcal{S}_{j}-\mathcal{S}_{j}'$ is symmetric and skew-symmetric, hence zero, and it is clear that
	$\mathcal{D}=\mathcal{D}'$.
\end{proof}

\mgap

Thinking geometrically, we are to find the projection of $\Sp$ onto 
$$\left\{\begin{pmatrix}0&\mathcal{S}_{2}\\\mathcal{S}_{3}&\mathcal{D}\end{pmatrix};\,\,\mathcal{S}_{2},\mathcal{S}_{3}\,\, \text{skew-symmetric}\right\}$$ along $\mathfrak{sp}(2n,\Rbb).$
That is, let $\calh=\begin{pmatrix}A&B\\C&D\end{pmatrix}\in \Sp$. Then
$$\pi(\calh)=\begin{pmatrix}0&\frac{1}{2}(B-B^{T})\\ \frac{1}{2}(C-C^{T})&A^{T}+D\end{pmatrix}.$$
Is every $\begin{pmatrix}0&\mathcal{S}_{2}\\ \mathcal{S}_{3}&\mathcal{D}\end{pmatrix}$ of this form?

\mgap

For a possible simplification, the map
\begin{align*}
	\mathcal{Y}:\,\,\Sp&\to\mathfrak{so}(2n,\Rbb)\\
	\calh&\mapsto \mmap(-\Js\calh)=\calh-\calh^{T}
\end{align*}
has the same image as $\mmap:\,\,\Sp\to\mathfrak{so}(2n,\Rbb)$ and may be easier to understand.

\mgap
\mgap

\important{Acknowledgments.}  Funding for this project was generously provided by Grinnell College, as part of its summer directed research program. The third author is currently a faculty member at California State University, Sacramento, but was a visiting faculty member at Grinnell College while the first draft of the paper was written; he thanks both institutions for their support. The authors are also grateful to the anonymous referee for a thoughtful reading of the paper and suggestions for its improvement.

\sgap

\important{Chris Hellmann}\\
\textit{Department of Mathematics and Statistics, Grinnell College, Grinnell, IA 50112}\\
\textit{hellmann@grinnell.edu}

\sgap

\important{Brennan Langenbach}\\
\textit{Department of Mathematics and Statistics, Grinnell College, Grinnell, IA 50112}\\
\textit{langenba@grinnell.edu}

\sgap

\important{Michael VanValkenburgh}\\
\textit{Department of Mathematics and Statistics, California State University, Sacramento,\\
Sacramento, CA 95819}\\
\textit{mjv@csus.edu}


\begin{thebibliography}{20}
	\bibitem{R:AbeMarsden} R. Abraham and J. E. Marsden, \textit{Foundations of Mechanics} (Second Edition), the Benjamin/Cummings Publishing Company, Inc., 1978.
	
	\bibitem{R:Arnold} V. I. Arnold, \textit{Mathematical Methods of Classical Mechanics,} Graduate Texts in Mathematics 60, Springer-Verlag, 1978. 

	\bibitem{R:Cannas} A. Cannas da Silva, \textit{Lectures on Symplectic Geometry,}
		Lecture Notes in Mathematics 1764, Springer-Verlag, 2001.
		
	\bibitem{R:DimSj} M. Dimassi and J. Sj\"{o}strand, \textit{Spectral Asymptotics in the Semiclassical Limit,} London Mathematical Society Lecture Notes Series 268, Cambridge University Press, 1999. 
	
	\bibitem{R:DuistFIO} H. Duistermaat, \textit{Fourier Integral Operators,} Birkh\"{a}user, 1996 edition.
	
	\bibitem{R:EverittMarkus} W. N. Everitt and L. Markus, Infinite dimensional complex symplectic spaces, \textit{Mem. Amer. Math. Soc.} 171 (2004), no. 810.
		
	\bibitem{R:FollandHAPS} G. Folland, \textit{Harmonic Analysis in Phase Space,} Annals of Mathematics Studies 122, Princeton University Press, 1989. 
	
	\bibitem{R:GrSj} A. Grigis and J. Sj\"{o}strand, \textit{Microlocal Analysis for Differential Operators,} London Mathematical Society Lecture Notes Series 196, Cambridge University Press, 1994.

	\bibitem{R:GuStGA} V. Guillemin and S. Sternberg, \textit{Geometric Asymptotics,} Mathematical Surveys and Monographs 14, American Mathematical Society, 1977.
	
	\bibitem{R:GuStPhy} V. Guillemin and S. Sternberg, \textit{Symplectic Techniques in Physics,} Cambridge University Press, 1983.
	
	\bibitem{R:HamRaysOne} Sir W. R. Hamilton, Theory of systems of rays, part first, \textit{Transactions of the Royal Irish Academy} \textbf{15} (1828), pp. 69-174.
	
	\bibitem{R:HamDynOne} Sir W. R. Hamilton, On a general method in dynamics, \textit{Philosophical Transactions of the Royal Society}, Part II for 1834, pp. 247-308.
	
	\bibitem{R:Schapira} P. Schapira, Conditions de positivit\'{e} dans une vari\'{e}t\'{e} symplectique complexe. Application \'{a} l'\'{e}tude des microfonctions, \textit{Annales scientifiques de l'\'{E}cole Normale Sup\'{e}rieure} \textbf{14.1} (1981) 121-139.
	
	\bibitem{R:SjSAM} J. Sj\"{o}strand, \textit{Singularit\'{e}s Analytiques Microlocales,} Ast\'{e}risque 95, 1982.
	
	\bibitem{R:TaoClGr} T. Tao,  The closed graph theorem in various categories, blog post from November 20, 2012, available at\\
	\url{http://terrytao.wordpress.com/2012/11/20/the-closed-graph-theorem-in-various-categories/}.
	
	\bibitem{R:Viterbo} C. Viterbo, Symplectic topology as the geometry of generating functions, \textit{Mathematische Annalen} Volume 292, Issue 1 (1992) 685-710.
	
	\bibitem{R:Weinstein} A. Weinstein, Symplectic geometry, \textit{Bull. Amer. Math. Soc. (N.S.)} \textbf{5}, Number 1 (1981) 1-13.
	
\end{thebibliography}
\end{document}